\documentclass{amsart}
\usepackage{fixltx2e}

\usepackage{tikz}

\usepackage{amsmath, amssymb}
\usepackage{amstext, amscd, amsfonts}
\usepackage{mathtools}
\usepackage{cancel}

\usepackage{varioref} 
\usepackage{url}

\newcommand{\Z}{\mathbb{Z}}

\newcommand{\cC}{\mathcal C}
\newcommand{\cD}{\mathcal D}

\newtheorem{thm}{Theorem}[section]
\newtheorem{thm*}{Theorem}
\newtheorem{theorem}{Theorem}[section]
\newtheorem{theorem*}{Theorem}

\newtheorem{corollary}[thm]{Corollary}

\newtheorem{lemma}[thm]{Lemma}

\theoremstyle{definition}

\newtheorem{defn*}{Definition}
\newtheorem{definition}{Definition}[section]
\newtheorem{definition-st}{Definition}

\newtheorem{example-st}{Example}

\newtheorem{rem-st}{Remark}

\newtheorem{remark-st}{Remark}

\numberwithin{equation}{section}
\numberwithin{figure}{section}

\newtheorem{cor*}{Corollary}
\newtheorem{prop*}{Proposition}

\newtheorem{obser-st}{Observation}

\theoremstyle{definition}

\newtheorem{exmp-st}{Example}

\newtheorem{exmps-st}{Examples}

\theoremstyle{remark}

\newtheorem{rems-st}{Remarks} 

\newtheorem{ack-st}{Acknowledgment}

\title{Relative equilibrium states and class degree}
\author{Jisang Yoo}
\address{Ajou University, Suwon, South Korea}
\email{(replace X with my last name) jisang.X.ac+rs12@gmail.com}
\begin{document}
\maketitle
\begin{abstract}
  Given a factor code $\pi$ from a shift of finite type $X$ onto a sofic shift $Y$, an ergodic measure $\nu$ on $Y$, and a function $V$ on $X$ with summable variation, we prove an invariant upper bound on the number of ergodic measures on $X$ which project to $\nu$ and maximize $h(\mu) + \int V d\mu$ among all measures in the fiber $\pi^{-1}(\nu)$. If $\nu$ is fully supported, this bound is the class degree of $\pi$. This generalizes a previous result for the special case of $V=0$.
\end{abstract}

\section{Introduction}

It is a classical result that given an irreducible shift of finite type $X$ there is a unique measure $\mu$ on $X$ that maximizes entropy $h(\mu)$ and that this unique measure, called the measure of maximal entropy, is an easily described Markov measure \cite{Parry-1964-intrinsic-markov-chain}. Also, given a real-valued function defined on $X$ with enough regularity, there is a unique measure on $X$, called the equilibrium state of $V$, that maximizes $h(\mu) + \int V d\mu$. Equilibrium states are more general than measure of maximal entropy: the equilibrium state of $0$ is the measure of maximal entropy.

We consider the relative case where a factor code $\pi:X \to Y$ from a shift of finite type $X$ to a sofic shift $Y$ is fixed and an ergodic measure $\nu$ is given. In the relative case, we restrict our attention to the measures in the fiber $\pi^{-1}(\nu)$. Even if $X$ is irreducible, there can be more than one measure that maximizes entropy among measures in $\pi^{-1}(\nu)$. These measures are called measures of relative maximal entropy. Petersen, Quas, and Shin proved that the number of ergodic measures of relative maximal entropy is always finite and gave an explicit upper bound~\cite{PQS-MaxRelEnt}. Allahbakhshi and Quas improved the upper bound to a conjugacy-invariant upper bound and introduced the notion of class degree~\cite{all2013classdegrelmaxent}. In the special case of $\nu$ with full support, their upper bound is equal to the class degree of the factor code. In the same paper, they proposed the conjecture that the class degree may also be the upper bound for the number of ergodic relative equilibrium states. Given a function $V$ on $X$ with summable variation, we prove the same conjugacy-invariant upper bound for the number of ergodic relative equilibrium states of $V$ over $\nu$.

In order to motivate parts of our proof and to explain the new main ingredient in the proof of our result, we explain shortly how previous results are proved. The previous result that the number of ergodic measures of relative maximal entropy is finite is proved in the following way. Suppose $\mu_1, \dots, \mu_{d+1}$ are distinct ergodic measures of relative maximal entropy over $\nu$ where $d$ is the number of letters (for $X$) that project to a fixed letter $b$ for $Y$ with $\nu(b)>0$. Form a relatively independent joining of the $d+1$ measures over $\nu$. Pigeonhole's principle then forces at least two, say $\mu_1,\mu_2$, of the $d+1$ measures to have the property $$\lambda(\{(x^{(1)},x^{(2)}): x^{(1)}_0 = x^{(2)}_0\}) > 0$$ where $\lambda = \mu_1 \otimes_\nu \mu_2$ is the relatively independent joining of the two measures over $\nu$. Since for every $(x^{(1)},x^{(2)})$ in some set of positive measure with respect to $\lambda$, there are infinitely many $i$ for which $x^{(1)}_i = x^{(2)}_i$, one can construct a point $x^{(3)}$ which is the result of splicing parts of $x^{(1)}$ or $x^{(2)}$ depending on the outcome of tossing a fair coin at every $i$ for which $x^{(1)}_i = x^{(2)}_i$. The probability distribution of the new point $x^{(3)}$ is a measure $\mu_3$ on $X$ which projects to $\nu$. The entropy of the new measure $\mu_3$ is then shown to be strictly greater than the entropy of $\mu_1$ or $\mu_2$, contradicting the initial assumption $\mu_1$ and $\mu_2$ are measures of relative maximal entropy. Therefore the number of ergodic measures of relative maximal entropy over $\nu$ cannot exceed $d$. The entropy increase of $\mu_3$ is shown with an application of Jensen's inequality.

The notion of class degree of a factor code is defined using an equivalence relation within fibers~\cite{all2013classdegrelmaxent}. Given a point $y$, the fiber $\pi^{-1}(y)$ is divided into finitely many components under the following equivalence relation: $x, x' \in \pi^{-1}(y)$ are equivalent if there is $x''$ in the same fiber that agrees with $x$ on $(-\infty,n]$ for a given arbitrary $n$ and with $x'$ on $[m,+\infty)$ for some $m>n$ and vice versa. The equivalence classes here are called transition classes over $y$. The number of transition classes over any transitive point $y \in Y$ is finite and same. This number is defined to be the class degree of the factor code. This generalizes the notion of the degree of a finite-to-one factor code: the common finite number of preimage points of any transitive point $y \in Y$ when $\pi$ is finite-to-one.

If we think of transition classes within fibers as the relative analogue of transitive components of a shift of finite type, it is natural to ask if the class degree bounds the number of measures of relative maximal measures and the number of relative equilibriums (over an ergodic $\nu$ with full support). 

The previous result that the number of ergodic measures of relative maximal measures is bounded by the class degree (when $\nu$ has full support) is proved in the following way. Suppose $\mu_1, \dots, \mu_{d+1}$ are distinct ergodic measures of relative maximal entropy over $\nu$ where $d$ is the class degree. As before, form a relatively independent joining of the $d+1$ measures over $\nu$ and apply pigeonhole's principle to conclude that for at least two measures, say $\mu_1, \mu_2$, we have $$\lambda(\{(x^{(1)},x^{(2)}): x^{(1)} \sim x^{(2)}\}) > 0$$ where $x^{(1)} \sim x^{(2)}$ means the two points are in the same transition class and  $\lambda = \mu_1 \otimes_\nu \mu_2$. Then the uniform conditional distribution property of measures of relative maximal entropy is used to show that this implies $$\lambda(\{(x^{(1)},x^{(2)}): x^{(1)}_0 = x^{(2)}_0\}) > 0$$ and a contradiction follows.

The proof of our result starts similarly by supposing that there are $d+1$ distinct ergodic relative equilibrium states of $V$. In our proof, we have to construct a new measure that satisfies a condition stronger than $$h(\mu_3) > h(\mu_1) \text{ or } h(\mu_3) > h(\mu_2)$$
namely
\begin{align*}
  h(\mu_3) &> h(\mu_1) + |\int V d\mu_1 - \int V d\mu_3| \text{ or} \\ h(\mu_3) &> h(\mu_2) + |\int V d\mu_2 - \int V d\mu_3|
\end{align*}

In other words, we need to construct a measure with the increase in entropy big enough that it overcomes the difference in integrals. An ingredient in our proof is an observation made in the following result by Antonioli~\cite{ant2014compen}. Given a relative equilibrium state $\mu$ of $V$ with summable variation, he showed that if $\mu$ does not have full support and $\nu$ has, then one can construct a new measure $\mu'$ by routing parts of a point in $X$ depending on outcomes of tossing a coin and the new measure has bigger $h(\mu') + \int V d\mu'$. If routing is done by using an $X$-block with zero measure, then it is known that the new measure has a bigger entropy~\cite{yoomaximal} (which proves that any measure of maximal relative entropy over $\nu$ has full support). Antonioli's new observation is that if a biased coin is used, then as the probability of coming up tails approaches zero, the difference in entropy (between the new measure $\mu'$ and the old measure $\mu$) dominates the difference in integral $\int V d\mu' - \int V d\mu$ (which proves that any relative equilibrium state of $V$ over $\nu$ has full support). The observation relies on restorability of the old point from the new point.

In our setting where we have a joining of two measures $\mu_1, \mu_2$, given two points $x^{(1)}, x^{(2)}$ (random points with its joint distribution being the joining), we have to form other points $x^{(3)}, x^{(4)}$ by alternating between parts of $x^{(1)}$ and $x^{(2)}$ in some way. The main difficulty in applying Antonioli's observation to our setting is that we cannot restore the old points $x^{(1)}, x^{(2)}$ from the new points $x^{(3)}, x^{(4)}$. But since $\mu_1, \mu_2$ are distinct ergodic measures, long blocks from $x^{(1)}$ are distinguishable from long blocks of $x^{(2)}$ with low probability of error. The rate of error goes to zero as the blocks become longer. The difficulty now is that we do not know enough about the speed of convergence of the error rate. Our main ingredient is in tossing a coin for every $N$'th occurrence of a fixed minimal transition block in order to work around this difficulty and $N$ is chosen in response to the speed of convergence of the error rate. This allows us to construct two new points in such a way that the increase in entropy dominates the difference in integral even if the speed of convergence of the error rate is slow. To enable this workaround, we prove some new results on the measure theoretic structure of infinite-to-one factor codes which are analogues of previous results on the topological structure.







\section{Background}
In this section, we introduce basic terminology and known results that will be used in our proof.

Throughout this paper, measures are always assumed to be probability measures. Shift spaces are assumed to be two-sided one-dimensional shift spaces.

A triple $(X,Y,\pi)$ is called a \emph{factor triple} if $\pi: X \to Y$ is a factor code from a SFT $X$ to a sofic shift $Y$. If a factor triple is such that $\pi$ is a 1-block factor code and $X$ is a 1-step SFT, then it is called a \emph{1-step 1-block factor triple}. Given a factor triple $(X,Y,\pi)$, there is a 1-step 1-block factor triple $(X', Y', \pi')$ that is topologically conjugate to $(X,Y,\pi)$~\cite{LM}.

\begin{definition}
  Given a factor triple $(X,Y,\pi)$ and an invariant measure $\nu$ on $Y$. An invariant measure $\mu$ on $X$ is called a measure of relative maximal entropy over $\nu$ if it projects to $\nu$ and its entropy is the biggest among all invariant measures on $X$ that projects to $\nu$.
\end{definition}
There is always at least one measure of relative maximal entropy over $\nu$. If $\nu$ is ergodic, then the ergodic decomposition of such $\mu$ decomposes it into ergodic measures of relative maximal entropy over $\nu$.

\begin{definition}
  Given a factor triple $(X,Y,\pi)$ and an invariant measure $\nu$ on $Y$ and a real-valued function on $X$ with summable variation. An invariant measure $\mu$ on $X$ is called a relative equilibrium state of $V$ over $\nu$ if it projects to $\nu$ and $h(\mu) + \int V d\mu$ is the biggest among all invariant measures on $X$ that projects to $\nu$.
\end{definition}
There is always at least one relative equilibrium state of $V$ over $\nu$. If $\nu$ is ergodic, then the ergodic decomposition of such $\mu$ decomposes it into ergodic relative equilibrium states of $V$ over $\nu$.

For more on the general theory of relative equilibrium states, see~\cite{Walters-Relative}.

We are not using any advanced probability theory, but in order to reduce verbosity of our arguments, we will borrow the language of random variables. Random variables here are defined to be almost everywhere defined measurable functions from a fixed Lebesgue space to Polish spaces. The notion of functions of a random variable, joint random variable, and probability distribution of a random variable are adopted.
\section{Class degree of a factor code}

The class degree of a factor code and the concept of transition blocks and minimal transition blocks are defined in \cite{all2013classdegrelmaxent}.

Given a 1-step 1-block factor triple $(X,Y,\pi)$ and an ergodic measure $\nu$ on $Y$, we say $(w,n,M)$ is a $\nu$-minimal transition block if it is a transition block with $\nu(w) > 0$ and if it has the smallest depth among all such transition blocks. If $\nu$ has full support, then the $\nu$-minimal transition blocks are exactly the minimal transition blocks.


\section{Class degree of an ergodic measure}

Given a factor triple $(X,Y,\pi)$ and an ergodic measure $\nu$ on $Y$, we define the class degree of $\nu$ to be a positive integer defined by the following result \cite{all2013classdegrelmaxent}. (Its proof does not use the assumption that $Y$ is irreducible)

\begin{theorem}\label{thm:classdeg}
  Let $(X,Y,\pi)$ be a factor triple and $\nu$ an ergodic measure on $Y$. Then $\nu$-almost every point of $Y$ has the same number of transition classes over it. We call this number the class degree of $\nu$ and denote it by $c_{\pi,\nu}$ or $c_{\nu}$. If $(X,Y,\pi)$ is a 1-step 1-block factor triple, then this number is equal to the depth of any $\nu$-minimal transition block. If $\nu$ is fully supported, then this number is equal to the class degree of $\pi$.
\end{theorem}


\section{Measure theoretic properties of transition classes}

We establish a measure theoretic analogue of a result in \cite{all2014structureclass}.
\begin{theorem}\label{them:uniqroute}
  Let $(X,Y,\pi)$ be a 1-step 1-block factor triple and $\nu$ an ergodic measure on $Y$. Let $\mu$ be an invariant measure on $X$ that projects to $\nu$. Let $(w,n,M)$ be a $\nu$-minimal transition block. Let $u$ be an X-block with $\mu(u) > 0$ that projects to $w$. Then $u$ is routable through a unique symbol in $M$ at time $n$.
\end{theorem}
\begin{proof}
Since $w$ is a transition block, $u$ is routable through at least one member of $M$ at time $n$. To show that $u$ is routable through at most one member of $M$ at time $n$, we suppose to the contrary that $u$ is routable through two distinct members $a^{(1)}$ and $a^{(2)}$ of $M = \{a^{(1)},a^{(2)}, \dots, a^{(d)}\}$ where $d \ge 2$ is the size of $M$. By Theorem~\ref{thm:classdeg}, $d$ is equal to the class degree of $\nu$.

By Poincare's recurrence theorem, for $\mu$-almost every point $x$ in the cylinder $[u] \subset X$, the block $u$ occurs infinitely many times to the right in $x$. And for $\mu$-almost every point $x \in X$, the point $\pi(x)$ has exactly $d$ transition classes over it. Therefore there exists a point $x \in X$ such that $u$ occurs infinitely many times to the right in $x$ and that $\pi(x)$ has exactly $d$ transition classes over it. Fix such a point $x$. Fix $d-1$ points $x^{(1)}, x^{(2)}, \dots, x^{(d-1)} \in X$ such that the $d$ points $x, x^{(1)}, x^{(2)}, \dots, x^{(d-1)}$ are in different transition classes over $\pi(x)$.

Since $u$ occurs infinitely many times to the right in $x$, we can choose positions $\{ [i_j, i_j + |w| -1] \}_{j \ge 1}$ such that $i_{j+1} > i_j + |w|$ and $x_{[i_j, i_j + |w| -1]} = u$.

For each $k$ and $j$, the block $x^{(k)}_{[i_j, i_j + |w| -1]}$ projects to $w$ and hence is routable through a symbol in $M$ at time $n$.

If there is $x^{(k)}$ such that $x^{(k)}_{[i_j, i_j + |w| -1]}$ is routable through $a^{(1)}$ or $a^{(2)}$ at time $n$ for infinitely many $j$, then the point $x^{(k)}$ is in the same transition class as $x$, which gives a contradiction.

Therefore there is $J \ge 1$ such that for each $j \ge J$ and for each $x^{(k)}$, the block $x^{(k)}_{[i_j, i_j + |w| -1]}$ is routable through a symbol, say $a^{(q(k,j))}$, in $M \setminus \{a^{(1)},a^{(2)}\}$ at time $n$. By the pigeonhole principle, for each $j \ge J$, there are distinct $k'_j, k''_j$ such that $q(k'_j,j) = q(k''_j, j)$. By applying the pigeonhole principle again, there are two distinct points $x^{(k')}, x^{(k'')}$ among the $d-1$ points such that $k' = k'_j$ and $k'' = k''_j$ for infinitely many $j \ge J$. The two points have the property that for infinitely many $j$, the blocks $x^{(k')}_{[i_j, i_j + |w| -1]}$ and $x^{(k'')}_{[i_j, i_j + |w| -1]}$ are routable through a common symbol at time $n$. This forces the two points to be in the same transition class, which gives a contradiction.
\end{proof}

We now introduce the notion of relative joining, of which relatively independent joining is an example. Given a factor triple $(X,Y,\pi)$, an invariant measure $\lambda$ on the product $X^2$ is called a (2-fold) \emph{relative joining} if for $\lambda$-almost every $(x,x')$ we have $\pi(x) = \pi(x')$. Let $p_1: X^2 \to X$ (resp. $p_2: X^2 \to X$) be the projection onto the first coordinate (resp. the second coordinate). Given a relative joining $\lambda$ on $X^2$, if $\mu_1$ and $\mu_2$ are invariant measures on $X$ such that $\mu_1 = p_1(\lambda)$ and $\mu_2 = p_2(\lambda)$, then we say $\lambda$ is a relative joining of $\mu_1$ and $\mu_2$. Given a relative joining $\lambda$ on $X^2$, if $\nu$ is an invariant measure on $Y$ such that $\nu = \pi\circ p_1(\lambda)$, then we say $\lambda$ is a relative joining over $\nu$.

If $\lambda$ is a relative joining of $\mu_1$ and $\mu_2$ over $\nu$, then $\pi(\mu_1) = \nu = \pi(\mu_2)$. Conversely, if $\mu_1, \mu_2$ are invariant measures on $X$ and if $\nu$ is an invariant measure on $Y$ such that $\pi(\mu_1) = \nu = \pi(\mu_2)$, then there is a relative joining of $\mu_1$ and $\mu_2$ over $\nu$, namely, the relatively independent joining.

We define and compare three subsets of $X^2$ given a 1-step 1-block factor triple $(X,Y,\pi)$. Let $D_1$ be the set of $(x,x') \in X^2$ such that $\pi(x) = \pi(x')$ and that $x,x'$ are in the same transition class over $\pi(x)$. We call this set the \emph{class diagonal} from the factor triple $(X,Y,\pi)$.
Let $D_2$ be the set of $(x,x') \in X^2$ such that $\pi(x) = \pi(x')$ and that $x,x'$ are routable through a common symbol at a common time.
Let $D_3$ be the set of $(x,x') \in X^2$ such that $\pi(x) = \pi(x')$ and that there is a point $z \in \pi^{-1}\pi(x)$ that is left asymptotic to $x$ and right asymptotic to $x'$ and a point $z' \in \pi^{-1}\pi(x)$ that is left asymptotic to $x'$ and right asymptotic to $x$.
The three sets $D_1, D_2, D_3$ are invariant Borel-measurable subsets of $X^2$ and we have $D_1 \subset D_2 \subset D_3$.

\begin{theorem}\label{thm:3d}
  Given a 1-step 1-block factor triple $(X,Y,\pi)$ and a relative joining $\lambda$ on $X^2$, we have $D_1 = D_2 = D_3 \pmod\lambda$.
\end{theorem}
\begin{proof}
It is enough to show that $D_3 \subset D_1 \pmod\lambda$.

Let $C$ be the set of pairs $(u,v)$ of $X$-blocks such that $\pi(u) = \pi(v)$ and that there is an $X$-block $w \in \pi^{-1}\pi(u)$ that starts with the same symbol as $u$ and ends with the same symbol as $v$ and a $X$-block $w' \in \pi^{-1}\pi(u)$ that starts with the same symbol as $v$ and ends with the same symbol as $u$.

For each $(u,v) \in C$, let $D_{(u,v)}$ be the set of $(x,x') \in X^2$ such that $\pi(x) = \pi(x')$ and that the $X^2$-block $(u,v)$ occurs in $(x,x')$. Then $D_3$ is the union of $D_{(u,v)}$.

For each $(u,v) \in C$, let $D'_{(u,v)}$ be the set of  $(x,x') \in X^2$ such that $\pi(x) = \pi(x')$ and that the $X^2$-block $(u,v)$ occurs infinitely many times to the right in $(x,x')$. By Poincare's recurrence theorem, $D_{(u,v)} = D'_{(u,v)} \pmod\lambda$.

It is easy to check that each $D'_{(u,v)}$ is a subset of $D_1$.
\end{proof}

A relative joining $\lambda$ on $X^2$ is called a \emph{class diagonal joining} if for $\lambda$-almost every $(x, x')$, the two points $x,x'$ are in the same transition class over the point $\pi(x) = \pi(x')$.

The following theorem is a measure theoretic analogue of another result in \cite{all2014structureclass}.
\begin{theorem}\label{thm:blockroute}
  Let $(X,Y,\pi)$ be a 1-step 1-block factor triple and $\nu$ an ergodic measure on $Y$. Let $\lambda$ be a class diagonal joining on $X^2$ over $\nu$. Let $(w,n,M)$ be a $\nu$-minimal transition block. Let $u,v$ be X-blocks that projects to $w$ such that $\lambda([u]\times [v]) >0$. Then the two blocks $u,v$ are routable through a common symbol in $M$ at time $n$.
\end{theorem}
\begin{proof}
Since $w$ is a transition block, $u$ is routable through a symbol in $M$, say $a$, at time $n$. Similarly, $v$ is routable through a symbol in $M$, say $b$, at time $n$. It is enough to show that $a=b$. Suppose $a \neq b$.

Let $C$ be the set as defined in the proof of Theorem~\ref{thm:3d}.

For each $(u'',v'') \in C$, let $D_{(u'',v'')}$ be the set of $(x,x') \in X^2$ such that $\pi(x) = \pi(x')$ and that the $X^2$-block $(u'',v'')$ occurs in $(x,x')$. Then $D_3$ is the union of $D_{(u'',v'')}$ when $(u'',v'')$ runs over the elements of $C$.

For each $(u'',v'') \in C$, let $D''_{(u'',v'')}$ be the set of  $(x,x') \in X^2$ such that $\pi(x) = \pi(x')$ and that the $X^2$-block $(u'',v'')$ occurs infinitely many times both in $(x,x')_{[0,\infty)}$ and in $(x,x')_{(-\infty,0]}$. By Poincare's recurrence theorem, $D_{(u'',v'')} = D''_{(u'',v'')} \pmod\lambda$.

Since $\lambda(D_1)=1$, we have $\lambda(D_3)=1$, but since $D_3 = \bigcup_{(u'',v'') \in C} D''_{(u'',v'')} \pmod\lambda$, there is $(u'',v'') \in C$ such that $\lambda([u]\times [v] \cap D''_{(u'',v'')}) > 0$. Fix such $(u'',v'') \in C$.

For each $(x,x') \in [u]\times [v] \cap D''_{(u'',v'')}$, the $X^2$-block $(u'',v'')$ occurs in $(x,x')_{[|w|, \infty)}$ and in $(x,x')_{(-\infty,-1]}$ while $(u,v)$ occurs between them. Therefore, there is $(\bar u, \bar v)$ with $\lambda([\bar u]\times [\bar v]) > 0$ such that $(u'',v'')$ occurs at the beginning and at the end of $(\bar u, \bar v)$ and that $(u,v)$ occurs at a position, say $[i, i+|w|-1]$, between them in $(\bar u, \bar v)$.

Since $\lambda([\bar u]\times [\bar v]) > 0$, we have $\pi(\bar u) = \pi(\bar v)$ and $\mu(\bar u) > 0$ where $\mu = p_1(\lambda)$. Let $\bar w = \pi(\bar u)$. Since $\bar w$ contains $w$ and $\nu(\bar w) > 0$, we can conclude that $(\bar w, i+n, M)$ is another $\nu$-minimal transition block.

The block $\bar u$ is routable through the symbol $a \in M$ at time $i+n$. Because $(u'',v'') \in C$ occurs at the beginning and at the end of $(\bar u, \bar v)$, the block $\bar u$ is routable through also $b \in M$ at time $i+n$. This contradicts Theorem~\ref{them:uniqroute}.
\end{proof}

We have the following pointwise statement.
\begin{corollary}\label{cor:pointroute}
  Let $(X,Y,\pi)$ be a 1-step 1-block factor triple and $\nu$ an ergodic measure on $Y$. Let $\lambda$ be a class diagonal joining on $X^2$ over $\nu$. Let $(w,n,M)$ be a $\nu$-minimal transition block. For $\lambda$-almost every $(x,x')$, we have that for each $i$ with $\pi(x)_{[i,i+|w|-1]} = w$, the two blocks $x_{[i,i+|w|-1]}$ and $x'_{[i,i+|w|-1]}$ are routable through a common symbol in $M$ at time $n$.
\end{corollary}
\begin{proof}
  For $\lambda$-almost every $(x,x')$, all $X^2$ blocks $(u,v)$ occurring in $(x,x')$ satisfy $\lambda([u]\times [v]) > 0$.
\end{proof}

\section{Relative entropy}

Given a probability space $(\Omega, \mathcal F, \mathbb P)$ and a measurable finite partition $\cC$ and a sub-$\sigma$-algebra $\mathcal D \subset \mathcal F$ and an event $A \in \mathcal F$ with $\mathbb P(A)>0$, we denote by $H(\cC | \cD | A)$ the conditional entropy of $\cC$ given $\cD$ with respect to the conditional measure on $A$. Given a discrete random variable $\hat x$ and a random variable $\hat y$ on $\Omega$ and an event $A \in \mathcal F$ with $\mathbb P(A)>0$, we denote by $H(\hat x| \hat y |A)$ the conditional entropy of $\hat x$ given $\hat y$ with respect to the conditional measure on $A$.

With three discrete random variables $\hat x, \hat y, \hat z$ and a positive event $A$, we have
$$H(\hat x | \hat y, \hat z | A) = \sum_z \mathbb P(\hat z=z | A) H(\hat x | \hat y | [Z=z] \cap A)$$
where z runs over values in the range of $\hat z$. (This follows easily by proving for the special case $A = \Omega$ first.)

If $A$ is an event measurable with respect to $\hat y$, then
\begin{align*}
  H(\hat x | \hat y) &= H(\hat x | \hat y, 1_A) \\
  &= \mathbb P(A) H(\hat x | \hat y | A) + \mathbb P(A^c) H(\hat x | \hat y | A^c)
\end{align*}
where $1_A$ is the indicator function of $A$. If $A$ is an event that is not measurable with respect to $\hat y$, then only the second equality from above is guaranteed.


\begin{lemma}\label{lem:goodE}
  Let $\hat x$ be a discrete random variable and E be an event that is measurable with respect to a random variable $\hat y$. Suppose there are $K+1$ Borel-measurable functions $f_0, \dots f_K$ such that $\hat x = f_0(\hat y)$ holds a.s. on the event $E^c$ and that $\hat x \in \{f_1(\hat y), \dots, f_K(\hat y) \}$ holds a.s. on the event $E$. Then
$$H(\hat x | \hat y) \le \Pr(E) \log K$$
\end{lemma}

For each $0 \le p \le 1$, denote $H_p = - p \log p - (1-p) \log (1-p) \ge 0$.

\begin{lemma}\label{lem:badE}
  Let $\hat x$ be a discrete random variable and E be an event. Let $\hat y$ be a random variable. Suppose there are $K+1$ Borel-measurable functions $f_0, \dots f_K$ such that $\hat x = f_0(\hat y)$ holds a.s. on the event $E^c$ and that $\hat x \in \{f_1(\hat y), \dots, f_K(\hat y) \}$ holds a.s. on the event $E$. Then
  $$H(\hat x | \hat y) \le \Pr(E) \log K + H_{\Pr(E)} $$
\end{lemma}

Given finite partitions $\alpha, \beta$ on a measure-theoretic dynamical system of finite entropy, the following quantities are all equal.
\begin{itemize}
\item $H(\alpha | \alpha_1^\infty \vee \beta_{-\infty}^\infty)$
\item $\lim_n H(\alpha | \alpha_1^n \vee \beta_{-\infty}^\infty)$
\item $\lim_n \frac1n H(\alpha_0^{n-1} | \beta_{-\infty}^\infty)$
\item the metric entropy of the factor system $\alpha_{-\infty}^{\infty} \vee \beta_{-\infty}^{\infty}$ minus the metric entropy of the factor system $\beta_{-\infty}^{\infty}$
\end{itemize}

We denote by $h(\alpha_{-\infty}^\infty | \beta_{-\infty}^\infty)$ this quantity. If a random variable $\hat x$ (resp. $\hat y$) generates $\alpha_{-\infty}^\infty$ (resp. $\beta_{-\infty}^\infty$) for some finite partition $\alpha$ (resp. $\beta$), then we write $h(\hat x | \hat y) = h(\alpha_{-\infty}^\infty | \beta_{-\infty}^\infty)$.

We have the following subadditive property of relative entropy.

\begin{align*}
  h(\alpha_{-\infty}^\infty | \gamma_{-\infty}^\infty)  &\le h(\alpha_{-\infty}^\infty \vee \beta_{-\infty}^\infty | \gamma_{-\infty}^\infty) \\
  &= h(\alpha_{-\infty}^\infty | \beta_{-\infty}^\infty \vee \gamma_{-\infty}^\infty) + h(\beta_{-\infty}^\infty | \gamma_{-\infty}^\infty)
\end{align*}


\section{Jump extension}

Throughout this section, let $\mu$ be an invariant measure on a subshift $X$, and $A$ a spanning subset of $X$ with respect to $\mu$, in other words,
$$\mu(\cup_{i \in \mathbb Z}\sigma^{i}(A)) = 1$$ and hence by Poincare's recurrence
$$\mu\{ x \in X : \sigma^ix \in A \text{ for bi-infinitely many } i  \}=1$$

Throughout this section, also let $\eta$ be an invariant measure on $C^{\Z}$ and assume $0 \not\in C$. Let $D$ be the disjoint union of $C$ and $\{0\}$. Then there is an extension $(X \times D^{\mathbb Z}, \bar\mu, \sigma)$ of the system $(X, \mu, \sigma)$ with the following properties.
\begin{itemize}
\item $\bar\mu$ is an invariant measure on $X \times D^{\mathbb Z}$ that projects to $\mu$. (This property is just another way of saying that $(X \times D^{\mathbb Z}, \bar\mu, \sigma)$ is an extension).
\item For $\bar\mu$-almost every $(x, t)$, for all $i \in \mathbb Z$, $\sigma^i x \in A$ if and only if $t_i \ne 0$.
\item If $q$ is a measurable function from $X$ to $\mathbb Z$ such that $\sigma^{q(x)}(x) \in A$ holds for $\mu$-almost every $x$, then $g_q(\bar\mu) = \mu\times\eta$ where $g_q$ is a $\bar\mu$-almost everywhere defined measurable function from $X \times D^{\mathbb Z}$ to $X \times C^{\mathbb Z}$ defined by $g_q(x,t) = (x, (t_{q_k(x)})_k)$, where $$ \dots < q_{-1}(x) < q_0(x) = q(x) < q_1(x) < q_2(x) < \dots$$ are all the coordinates $i$ for which $\sigma^i(x) \in A$.
\end{itemize}

We call the extension $(X \times D^{\mathbb Z}, \bar\mu, \sigma)$ (or just the measure $\bar\mu$) the \emph{jump extension} of $(X, \mu, \sigma)$ with respect to $A$ and $\eta$.

\begin{theorem}\label{thm:jumpentropy}
  The entropy of the jump extension is $$h(\bar\mu) = h(\mu) + \mu(A)h(\eta)$$
\end{theorem}

\begin{lemma}\label{lem:t0}
  Let $C'$ be a subset of $C$ and let $B$ be a measurable subset of $X$. Then
  $$\bar\mu\{(x,t): x \in B, t_0 \in C'\} = \mu(B \cap A)\eta([C'])$$
  where $[C']$ denotes the cylinder $\{z \in C^\Z: z_0 \in C'\}$.
\end{lemma}
\begin{proof}
  For $\mu$-almost every $x$, define $q(x)$ to be the smallest nonnegative integer with $\sigma^{q(x)}(x) \in A$. Note that
  $$ \{(x,t): t_0 \neq 0\} = \{(x,t): x \in A \} = \{(x,t): q(x) = 0\} \pmod{\bar\mu}$$
  So we can conclude
  $$\{(x,t): x \in B, t_0 \in C'\} = \{(x,t): g_q(x,t) \in (B\cap A)\times [C']\} \pmod{\bar\mu}$$
\end{proof}
As a special case, we get the following corollary.
\begin{corollary}\label{lem:t0special}
  Let $C'$ be a subset of $C$. Then
  $$\bar\mu\{(x,t): t_0 \in C'\} = \mu(A)\eta([C'])$$
\end{corollary}
  

\section{Proof of the main theorem}

\begin{lemma}
  Let $(X,Y,\pi)$ be a factor triple and $\nu$ an ergodic measure on $Y$. Let $V$ be a function on $X$ with summable variation. Let $\lambda$ be a class diagonal joining of distinct ergodic measures $\mu_1,\mu_2$ over $\nu$. Then there is another relative joining $\lambda'$ on $X^2$ over $\nu$ such that
$$ h(\lambda') + \mu_1'(V) + \mu_2'(V) > h(\lambda) + \mu_1(V) + \mu_2(V) $$
where $\mu'_1 = p_1(\lambda')$ and $\mu'_2 = p_2(\lambda')$. 
\end{lemma}
\begin{proof}

For each $(N,p) \in \mathbb N \times (0, \frac12)$, first define a (non-invariant) measure $\eta^o = \eta^o_{(N,p)}$ on $\{1,2,3\}^{\mathbb Z}$: for each $i \in \Z$, $\eta^o([1]_{iN}) = 1-p$, $\eta^o([2]_{iN}) = p$, $\eta^o([3]_{iN}) = 0$ and for each $k$ not a multiple of $N$,  $\eta^o([3]_k) = 1$ and the measure $\eta^o$ makes each coordinate independent. Define the invariant measure $\eta = \eta_{(N,p)}$ on $\{1,2,3\}^{\mathbb Z}$ by:
$$\eta = \frac1N \sum_{k=0}^{N-1} \sigma^k(\eta^o)$$

The invariant measure $\eta_{(N,p)}$ satisfies the following properties.
\begin{itemize}
\item $\eta$-almost every point is concatenation of blocks of length $N$ that are either $13^{N-1}$ or $23^{N-1}$
\item Its entropy is $h(\eta) = \frac1N H_p$
\item $\eta(1) = \frac{1-p}{N}$
\item $\eta(2) = \frac{p}{N}$
\item $\eta(13^{N-1}2) = \frac{p(1-p)}{N} = \eta(23^{N-1}1)$
\end{itemize}

We may assume $(X,Y,\pi)$ is a 1-step 1-block factor triple. Let $(w,n,M)$ be a $\nu$-minimal transition block. Let $(N,p) \in \mathbb N \times (0, \frac12)$ be such that $N > |w|$. The value of $(N,p)$ will be determined later.

Let $\bar\lambda$ be the jump extension of $\lambda$ with respect to $(\pi\circ p_1)^{-1}[w]$ and $\eta_{(N,p)}$.
We can form the jump extension because $(\pi\circ p_1)^{-1}[w]$ is spanning: in fact, $\lambda$-almost every point visits $(\pi\circ p_1)^{-1}[w]$ with frequency given by $\nu(w) > 0$ because $\nu$ is ergodic.
$\bar\lambda$ is an invariant measure on $\Omega = X^2 \times \{0,1,2,3\}^{\mathbb Z}$. The measure-theoretic dynamical system $(\Omega, \bar\lambda, \sigma)$ is the ambient probability space on which we will build our random variables.

Let $\hat x, \hat x', \hat t$ be random variables defined on $(\Omega, \bar\lambda)$ by
$$ \hat x(x,x', t) = x $$
$$ \hat x'(x,x', t) = x' $$
$$ \hat t(x,x',t) = t $$

Since the distribution of the joint random variable $(\hat x, \hat x')$ is the relative joining $\lambda$, we can define another random variable $\hat y = \pi(\hat x) = \pi(\hat x')$ which has distribution $\nu$.
The jump extension ensures that for each $i$, the event $\hat t_i >0$ is the same as the event $\sigma^i(\hat y) \in [w]$. In other words, $\hat t$ is a sequence in which nonzero symbols occur exactly where the word $w$ occurs in $\hat y$.

We have so far four random variables: $\hat x, \hat x', \hat t, \hat y$. We now want to construct two more random variables $\hat z, \hat z'$ such that $\pi(\hat z) = \hat y = \pi(\hat z')$ and they will be formed by taking some segments from $\hat x, \hat x'$ in some way. We define $\hat z$ first. It will be defined in such a way that $\hat z$ is a function of $\hat x, \hat x', \hat t$. Occurrence of the symbol 1 in $\hat t$ will mean: take from the first path, namely, $\hat x$. The symbol 2 will mean: take from the second path, namely, $\hat x'$. The other symbols 3 and 0 have no meaning.

The point $\hat z(x,x',t) \in X$ is defined for $\bar\lambda$-almost every $(x,x',t)$ in the following way. Let $\dots i_{-1} < i_0 < i_1 < \dots$ be all the places where 1 or 2 occurs in $t$. (One can think of each $i_j$ to be a integer-valued function defined almost everywhere on $\Omega$ if preferred) Note that $i_{j+1} - i_j \ge N > |w|$ holds for each $j$ (almost everywhere) because if we remove zeros from the block $t_{[i_j, i_{j+1}-1]}$ we would get either $13^{N-1}$ or $23^{N-1}$. This means that we can divide the region $[i_j, i_{j+1}-1]$ into two subregions $[i_j, i_j+|w|-1]$ and $[i_j+|w|, i_{j+1}-1]$.

We define $\hat z(x,x',t)$ for the latter type of subregions first. The value of $\hat z$ on those subregions are copied from $x$ or $x'$ depending on what $t$ tells at $i_j$, in other words:
\[
\hat z(x,x',t)_{[i_j+|w|, i_{j+1}-1]} =
\begin{cases}
  x_{[i_j+|w|, i_{j+1}-1]} & \text{if } t_{i_j} = 1 \\
  x'_{[i_j+|w|, i_{j+1}-1]} & \text{if } t_{i_j} = 2
\end{cases}
\]

For the former type of subregions, note that for each of such subregion, the block $w$ appears in $\hat y(x,x',t)$ at that subregion. Since $\lambda$ is class diagonal, Corollary~\ref{cor:pointroute} ensures that for each of these subregions, the two blocks from $x, x'$ at that subregion are routable through a common symbol. Theorem~\ref{thm:blockroute} ensures that for each $X^2$-block $(u,v)$ that projects to $w$ such that $\lambda([u]\times [v]) >0$, one can choose an $X$-block $r^{12}(u,v)$ that projects to $w$ and starts with the symbol $u_0$ and ends with the symbol $v_{|w|-1}$. We also choose $r^{21}(u,v)$ that projects to $w$ and starts with the symbol $v_0$ and ends with the symbol $u_{|w|-1}$. We also define $r^{11}(u,v) = u$ and $r^{22}(u,v) = v$.

Now define $\hat z(x,x',t)$ for the former type of subregions by using the functions $r^{11},r^{12},r^{21},r^{22}$ depending on what $t$ is telling at $i_{j-1}$ and $i_j$, in other words:
\[
\hat z(x,x',t)_{[i_j, i_j+|w|-1]} = r^{t_{i_{j-1}} t_{i_j}}(x_{[i_j, i_j+|w|-1]}, x'_{[i_j, i_j+|w|-1]})
\]

It is easy to check that for $\bar\lambda$-almost every $(x,x',t)$, the point $\hat z(x,x',t)$ is well defined and is a point in $X$. As a random variable, one can also check that $\pi(\hat z) = \hat y$.

Define another random variable $\hat z'$ in much the same way as $\hat z$ except this time the meaning of the symbols 1 and 2 are swapped: the symbol 1 now means taking from the second path and 2 means taking  from the first path. $\hat z'$ is in some sense dual to $\hat z$. It is easy to check that the joint random variable $(\hat z, \hat z'): \Omega \to X^2$ as a function is shift-commuting, therefore the distribution of $(\hat z, \hat z')$ is an invariant measure on $X^2$, which we denote by $\lambda'$. This measure $\lambda'$ is a relative joining over $\nu$ because $\pi(\hat z) = \hat y = \pi(\hat z')$.

We have the following four equality or inequalities:
the inequality holds because $h(\hat t, \hat x, \hat x' | \hat z, \hat z') = h(\hat t, \hat x, \hat x', \hat z, \hat z') - h(\hat z, \hat z')$ and the second-to-last equality holds because it is the entropy of the jump extension.
\begin{align*}
  h(\lambda') &= h(\hat z, \hat z') \\
  h(\hat z, \hat z') + h(\hat t, \hat x, \hat x' | \hat z, \hat z') &\ge h(\hat t, \hat x, \hat x') \\
  h(\hat t, \hat x, \hat x') &= h(\hat x,\hat x') + \Pr(\hat t_0 > 0) h(\eta) \\
  h(\hat x, \hat x') &= h(\lambda)
\end{align*}

So we can conclude
$$ h(\lambda') - h(\lambda) \ge \Pr(\hat t_0 > 0) h(\eta) - h_0$$
where
$$ h_0 := h(\hat t, \hat x, \hat x' | \hat z, \hat z')$$

We want to bound $h_0$ from above. We divide it into $ h_0 = h_1 + h_2$ where  $$ h_1 = h(\hat t | \hat z, \hat z')$$
and $$ h_2 = h(\hat x, \hat x' | \hat t, \hat z, \hat z') $$

We obtain an upper bound for $h_1$ first. To do that, we introduce two more random variables $\hat t'$ and $\hat t''$.

$$ \hat t'_i =
\begin{cases}
  \hat t_i & \text{when } \hat t_i = 0, 3 \\
  4 & \text{when } \hat t_i = 1,2
\end{cases}
$$

The random variable $\hat t'$ captures partial information of $\hat t$ by not distinguishing 1 and 2.

$$ \hat t''_i =
\begin{cases}
  0 & \text{when } \hat t_i = 0 \\
  1 & \text{when } \hat t_i > 0
\end{cases}
$$

The random variable $\hat t''$ captures partial information of $\hat t$ that corresponds to where zeroes occur in $\hat t$ and where nonzeros occur. The following three events are equivalent mod $\bar\lambda$:
\begin{align*}
  \hat t''_i = 1 \\ \hat t'_i > 0 \\ \sigma^i(\hat y) \in [w]
\end{align*}

Note that $\hat y$ determines $\hat t''$. Also, $\hat t$ determines $\hat t'$ which in turn determines $\hat t''$.

We decompose $h_1$ into
$$ h_1 \le h(\hat t' | \hat z, \hat z') + h(\hat t | \hat t', \hat z, \hat z')$$

Since $\hat z$ determines $\hat y$ which in turn determines $\hat t''$, we have the following bound for the first term
$$ h(\hat t' | \hat z, \hat z') \le h(\hat t' | \hat t'')$$
but there is only $N$ possible values for $\hat t'$ given the value of $\hat t''$, therefore $h(\hat t' | \hat t'') = 0$ and we have
$$ h(\hat t' | \hat z, \hat z') \le h(\hat t' | \hat t'') = 0$$
and so
$$ h_1 \le h(\hat t | \hat t', \hat z, \hat z')$$

Therefore
\begin{align*}
  h_1 &\le H(\hat t_0 | \hat t_{[1,\infty)}, \hat t', \hat z, \hat z') \\
  &\le H(\hat t_0 | \hat t'_0, \hat z, \hat z') \\
  &= \Pr(\hat t'_0 = 4) H(\hat t_0 | \hat t'_0, \hat z, \hat z' | \hat t'_0 = 4)\\
  & \quad + \Pr(\hat t'_0 \neq 4) H(\hat t_0 | \hat t'_0, \hat z, \hat z' | \hat t'_0 \neq 4)\\
  &\le \Pr(\hat t'_0 = 4) H(\hat t_0 | \hat z, \hat z' | \hat t'_0 = 4)\\
  & \quad + \Pr(\hat t'_0 \neq 4) H(\hat t_0 | \hat t'_0 | \hat t'_0 \neq 4)\\
  &= \Pr(\hat t'_0 = 4) H(\hat t_0 | \hat z, \hat z' | \hat t'_0 = 4)
\end{align*}
where the last equality holds because $H(\hat t_0 | \hat t'_0 | \hat t'_0 \neq 4) = 0$ which is because $\hat t'_0$ determines $\hat t_0$ given the event $\hat t'_0 \neq 4$.

So we have
$$ h_1 \le \Pr(\hat t'_0 = 4) H^* $$
where
$$ H^* = H(\hat t_0 | \hat z, \hat z' | \hat t'_0 = 4) $$

We want to obtain an upper bound on $H^*$ such that it approaches 0 as $N \to 0$ and does not depend on $p$.

For convenience of further calculation, we let $J = [|w|, N-1]$ which depends on $N$ but not on $p$. Note that given the event $\hat t'_0 = 4$, the value of $(\hat z, \hat z')_J$ is either $(\hat x, \hat x')_J$ or $(\hat x', \hat x)_J$ depending on whether $\hat t_0$ is 1 or 2. Therefore, given the event $\hat t'_0 = 4$ and the event $(\hat x, \hat x')_J \in G_1 \times G_2$ where $G_1$ and $G_2$ are disjoint sets of blocks that we will define later, the value of $(\hat z, \hat z')_J$ determines the value of $\hat t_0$ (by just looking at which one of $G_1$ and $G_2$ the block $\hat z_J$ belongs to).

To define $G_1, G_2$, first choose $a$ to be an $X$-block such that $\mu_1(a) \neq \mu_2(a)$ and let $d = |\mu_1(a) - \mu_2(a)| > 0$. Such a block exists because $\mu_1$ and $\mu_2$ are assumed to be distinct. Let $G_1$ be the set of all $X$-blocks $b$ of length $|J| = N - |w|$ such that $$| D(a|b) - \mu_1(a) | < \frac{d}{2}$$ where $D(a|b)$ denotes the frequency of $a$ in $b$. Similarly, let $G_2$ to be the set of all $X$-blocks $b$ of length $|J|$ such that $$| D(a|b) - \mu_2(a) | < \frac{d}{2}$$

It is clear that the two sets $G_1, G_2$ are disjoint. By Lemma~\ref{lem:badE} we have
\begin{align*}
  H^* &\le H(\hat t_0 | (\hat z, \hat z')_J | \hat t'_0 = 4) \\
  &\le P^* \log 2 + H_{P^*}
\end{align*}
where $P^*$ denotes the conditional probability given by
$$ P^* = \Pr((\hat x, \hat x')_J \not\in G_1 \times G_2 |  \hat t'_0 = 4)$$

We want to show that $P^*$ is a quantity that goes to 0 as $N \to \infty$ and does not depend on $p$.

Write
$$ P^* = \frac{\Pr((\hat x, \hat x')_J \not\in G_1 \times G_2,  \hat t'_0 = 4)}{\Pr(\hat t'_0 = 4)}
$$
and apply Lemma~\ref{lem:t0} and Lemma~\ref{lem:t0special} to the numerator and the denominator to get
$$ P^* = \frac{\lambda(F_J)}{\nu(w)}$$
where $F_J \subset X^2$ denotes the set of $(x,x')$ such that $(x,x')_J \not\in G_1 \times G_2$ and $\pi(x) \in [w]$. The set $F_J$ depends on $J$ which in turn depends on $N$ but the set does not depend on $p$. It is easy to show, using the mean ergodic theorem applied to ergodic $\mu_1$ and $\mu_2$, that $\lim_N\lambda(F_J) = 0$.
Therefore $P^*$ (and hence $H^*$ too) is a quantity that does not depend on $p$ and goes to $0$ when $N \to \infty$. Denote $H^*$ by $H^*(N)$ to express its dependency on the parameter $N$. We showed that
$$ h_1 \le \Pr(\hat t'_0 = 4) H^*(N) $$
where $H^*(N)$ is a quantity that does not depend on $p$ and that $\lim_N H^*(N) = 0$.

Next we want to obtain an upper bound for $$ h_2 = h(\hat x, \hat x' | \hat t, \hat z, \hat z') $$

For $\lambda$-almost every $(x,x')$, let $q = q(x,x')$ be the smallest nonnegative number such that $\sigma^q \pi(x) \in [w]$ and let
$$\dots < q_{-1} < q_0 = q < q_1 < q_2 < \dots$$ be all the coordinates $i$ for which $\sigma^i \pi(x) \in [w]$.

Let $\hat q_k = q_k(\hat x, \hat x')$. Each $\hat q_k$ is an integer-valued random variable. Using them, define
$$\hat u = (\hat t_{\hat q_k})_{-N \le k \le 0}$$

The random variable $\hat u$ takes values in $\{1,2,3\}^{N+1}$ and the probability of the event $\hat u = u$ for each block $u$ is given by $\Pr(\hat u =u) = \eta(u)$.

Define the two events
\begin{align*}
  S_{12} &= [\hat y \in [w], \hat u = 13^{N-1}2] \\
  S'_{12} &= \bigcup_{0 \le k < |w|} \sigma^{k}(S_{12})
\end{align*}

The event $S'_{12}$ represents the event of the coordinate 0 falling to one of the subregions where we used the function $r^{12}$. Define $S_{21}$ and $S'_{21}$ similarly, with $23^{N-1}1$ in place of $13^{N-1}2$. Note that the four events we just defined are measurable with respect to $\hat t$. This allows us to use Lemma~\ref{lem:goodE} to say
\begin{align*}
  h_2 &\le H((\hat x, \hat x')_0 | \hat t, \hat z, \hat z') \\
  &\le \Pr(S'_{12}\cup S'_{21}) \log(C_0^2)
\end{align*}
where $C_0$ is the number of letters used in the SFT $X$.

We want to estimate $\Pr(S'_{12}\cup S'_{21})$ now.
\begin{align*}
  \Pr(S'_{12}) &\le |w|\Pr(S_{12}) \\
  &= |w| \cdot \nu(w) \cdot \eta(13^{N-1}2)\\
  &= |w| \cdot \nu(w) \cdot \frac{p(1-p)}{N}
\end{align*}

So we have $$\Pr(S'_{12}\cup S'_{21}) \le C_1 \cdot \frac{p}{N}$$ where $C_1$ is some constant depending on $w$ but not on $N$ or $p$.

It remains to estimate $\mu'_1(V) + \mu'_2(V) - \mu_1(V) - \mu_2(V)$ (we denote its absolute value by $h_3$ for later reference) which is the expectation of the real-valued random variable:
$$V(\hat z) + V(\hat z') - V(\hat x) - V(\hat x')$$

By using the same argument as in \cite{ant2014compen}, or alternatively by moving the calculation to the derivative system induced on $S'_{12}\cup S'_{21}$, one can show
$$ h_3 \le \Pr(S'_{12}\cup S'_{21}) \cdot C_2$$
where $C_2$ is some constant depending on $|w|$ and $V$ but not on $N$ or $p$.

We obtained upper bounds for all relevant quantities to estimate:
$$\Delta := (h(\lambda') + \mu_1'(V) + \mu_2'(V)) - (h(\lambda) + \mu_1(V) + \mu_2(V))$$
which is greater than or equal to
\begin{align*}
  & h(\lambda') - h(\lambda) - h_3\\
  &\ge \Pr(\hat t_0 > 0) h(\eta) - h_0 - h_3\\
  &\ge \Pr(\hat t_0 > 0) h(\eta) - h_1 - h_2 - h_3\\
  &\ge \Pr(\hat t_0 > 0) h(\eta)\\
  & \quad - \Pr(\hat t'_0 = 4) H^*(N) \\
  & \quad - \Pr(S'_{12}\cup S'_{21}) \log(C_0^2) \\
  & \quad - \Pr(S'_{12}\cup S'_{21}) \cdot C_2 \\
  &\ge \nu(w) \cdot \frac{H_p}{N}\\
  & \quad - \frac{\nu(w)}{N} \cdot H^*(N) \\
  & \quad - C_1 \cdot \frac{p}{N} \cdot \log(C_0^2) \\
  & \quad - C_1 \cdot \frac{p}{N} \cdot C_2 \\
\end{align*}

By choosing appropriate constants $C_3, C_4, C_4$ that does not depend on $N$ or $p$, we have
$$ \Delta \ge \frac{C_3 \cdot H_p - C_4 \cdot H^*(N) - C_5 \cdot p}{N}$$

Now we determine $(N,p)$. Choose $p$ to be be small enough that
$$ C_3 \cdot H_p - C_5 \cdot p > 0$$
and then choose $N$ to be large enough that
$$  C_4 \cdot H^*(N) < C_3 \cdot H_p - C_5 \cdot p$$
We chose $(N,p)$ so that
$$ \Delta > 0$$

\end{proof}

\begin{corollary}
    Let $(X,Y,\pi)$ be a factor triple and $\nu$ an ergodic measure on $Y$. Let $V$ be a function on $X$ with summable variation. Let $\lambda$ be a relative joining of distinct ergodic measures $\mu_1,\mu_2$ over $\nu$ such that $\lambda(D_1) > 0$ where $D_1$ is the class diagonal. Then there is another relative joining $\lambda'$ on $X^2$ over $\nu$ such that
$$ h(\lambda') + \mu_1'(V) + \mu_2'(V) > h(\lambda) + \mu_1(V) + \mu_2(V) $$
where $\mu'_1 = p_1(\lambda')$ and $\mu'_2 = p_2(\lambda')$. 
\end{corollary}
\begin{proof}
  We may assume $0 < p := \lambda(D_1) < 1$. We can decompose $\lambda$ into convex combination of two invariant measures:
$$ \lambda = p \lambda_1 + (1-p) \lambda_2 $$
where $\lambda_1(D_1) = 1$ and then both $\lambda_i$ are relative joinings of $\mu_1, \mu_2$ over $\nu$ because $\mu_1, \mu_2, \nu$ are assumed ergodic. By the previous lemma,  there is a  relative joining $\lambda'_1$ over $\nu$ such that
$$h(\lambda'_1) + (p_1(\lambda'_1))(V) + (p_2(\lambda'_1))(V) > h(\lambda_1) + \mu_1(V) + \mu_2(V) $$

We write
$$\lambda' = p \lambda'_1 + (1-p) \lambda_2$$
then $\lambda'$ is a relative joining over $\nu$.

It is an easy check that $\lambda'$ satisfies the strict inequality in the conclusion.
\end{proof}

\begin{corollary}
  Let $(X,Y,\pi)$ be a factor triple and $\nu$ an ergodic measure on $Y$. Let $V$ be a function on $X$ with summable variation. Let $\lambda$ be a relatively independent joining of distinct ergodic measures $\mu_1,\mu_2$ over $\nu$ where $\mu_1, \mu_2$ are both relative equilibrium states of $V$ over $\nu$. Then $\lambda(D_1) = 0$.
\end{corollary}
\begin{proof}
Suppose $\lambda(D_1) > 0$ instead. The previous corollary then applies to produce another relative joining $\lambda'$ on $X^2$ over $\nu$ such that
$$ h(\lambda') + \mu_1'(V) + \mu_2'(V) > h(\lambda) + \mu_1(V) + \mu_2(V) $$
where $\mu'_1 = p_1(\lambda')$ and $\mu'_2 = p_2(\lambda')$. Note that $\mu_1, \mu_2, \mu'_1, \mu'_2$ all project to $\nu$.

We have
\begin{align*}
  h(\lambda) + \mu_1(V) + \mu_2(V) &= h(\mu_1|\nu)+h(\mu_2|\nu)+h(\nu) + \mu_1(V) + \mu_2(V) \\
  &= (h(\mu_1|\nu)+\mu_1(V)) + (h(\mu_2|\nu)+\mu_2(V)) + h(\nu)  \\
  &\ge (h(\mu'_1|\nu)+\mu'_1(V)) + (h(\mu'_2|\nu)+\mu'_2(V)) + h(\nu) \\
  &= h(\mu'_1|\nu)+h(\mu'_2|\nu)+h(\nu) + \mu'_1(V) + \mu'_2(V) \\
  &\ge  h(\lambda') + \mu'_1(V) + \mu'_2(V)
\end{align*}
which contradicts our initial strict inequality.
\end{proof}

\begin{theorem}
  Let $(X,Y,\pi)$ be a factor triple and $\nu$ an ergodic measure on $Y$. Let $V$ be a function on $X$ with summable variation. The number of ergodic relative equilibrium states of $V$ over $\nu$ is less than or equal to the class degree of $\nu$.
\end{theorem}
\begin{proof}
  Suppose $d$ is the class degree of $\nu$ and that $\mu_1, \dots, \mu_{d+1}$ are $d+1$ distinct ergodic relative equilibrium states of $V$ over $\nu$. Form the $(d+1)$-fold relatively independent joining of these $d+1$ measures over $\nu$. The fact that there are only $d$ transition classes over $\nu$-almost every $y$ ensures the existence of distinct $i, j$ such that the projection of the $(d+1)$-fold joining to $i, j$ violates the previous corollary.
\end{proof}

\begin{corollary}
  Let $(X,Y,\pi)$ be a factor triple and $\nu$ an ergodic measure on $Y$ with full support. Let $V$ be a function on $X$ with summable variation. The number of ergodic relative equilibrium states of $V$ over $\nu$ is less than or equal to the class degree of $\pi$.
\end{corollary}
\begin{proof}
  Since $\nu$ has full support, the class degree of $\nu$ is the class degree of $\pi$.
\end{proof}



\bibliographystyle{amsplain}
\bibliography{my-reference} 

\end{document}